\documentclass[12pt]{article}

\usepackage[centertags]{amsmath}
\usepackage{amsfonts}
\usepackage{amssymb}
\usepackage{amsthm}
\usepackage{graphics,graphicx}
\usepackage{newlfont}
\usepackage{mathtools}\usepackage{latexsym}
\usepackage[usenames]{color}
\usepackage{amsthm}
\usepackage{hyperref}
\usepackage{amsmath} 
\usepackage{amsfonts} 
\usepackage{amssymb}
\usepackage{graphicx}
\usepackage{slashed}
\usepackage{fouridx}

\theoremstyle{plain} \newtheorem{defi}{Definition}[section]
\theoremstyle{plain} \newtheorem{prop}{Proposition}[section]
\theoremstyle{plain} \newtheorem{theor}{Theorem}[section] 
\theoremstyle{plain} \newtheorem{lemma}{Lemma}[section]
\theoremstyle{plain} \newtheorem{corol}{Corollary} 
\theoremstyle{remark} \newtheorem{rem}{Remark}
\theoremstyle{plain} 
\theoremstyle{remark}

\begin{document}
\begin{center}
\Large{\textbf{On the nontrivial zeros of the Dirichlet eta function}}\\ 
~\\

\large{Vladimir Garc\'{\i}a-Morales}\\

\normalsize{}
~\\

Departament de F\'{\i}sica de la Terra i Termodin\`amica\\ Universitat de Val\`encia, \\ E-46100 Burjassot, Spain
\\ garmovla@uv.es
\end{center}
\small{We construct a two-parameter complex function $\eta_{\kappa \nu}:\mathbb{C}\to \mathbb{C}$, $\kappa \in (0, \infty)$, $\nu\in (0,\infty)$ that we call a holomorphic nonlinear embedding and that is given by a double series which is absolutely and uniformly convergent on compact sets in the entire complex plane. The function  $\eta_{\kappa \nu}$ converges to the Dirichlet eta function $\eta(s)$ as $\kappa \to \infty$. We prove the crucial property that, for sufficiently large $\kappa$, the function $\eta_{\kappa \nu}(s)$ can be expressed as a linear combination $\eta_{\kappa \nu}(s)=\sum_{n=0}^{\infty}a_n(\kappa) \eta(s+2\nu n)$ of horizontal shifts of the eta function (where $a_{n}(\kappa) \in \mathbb{R}$ and $a_{0}=1$) and that, indeed, we have the inverse formula $\eta(s)=\sum_{n=0}^{\infty}b_n(\kappa) \eta_{\kappa \nu}(s+2\nu n)$ as well (where the coefficients $b_{n}(\kappa) \in \mathbb{R}$ are obtained from the $a_{n}$'s recursively). By using these results and the functional relationship of the eta function, $\eta(s)=\lambda(s)\eta(1-s)$, we sketch a proof of the Riemann hypothesis which, in our setting, is equivalent to the fact that the nontrivial zeros $s^{*}=\sigma^{*}+it^{*}$ of $\eta(s)$ (i.e. those points for which $\eta(s^{*})=\eta(1-s^{*})=0)$ are all located on the critical line $\sigma^{*}=\frac{1}{2}$. 
}
\noindent  
~\\
 \pagebreak

\section{Introduction}

Let $s:=\sigma+it$ be a complex number. The Dirichlet eta function $\eta(s)$, also called alternating zeta function, is given in the half plane $\sigma >0$ by the conditionally convergent series \cite{T}
\begin{equation}
\eta(s)=\sum_{m=1}^{\infty}\frac{(-1)^{m-1}}{m^s} \label{eta1}
\end{equation} 
which is absolutely convergent for $\sigma>1$. Hardy gave a simple proof of the fact that the eta function satisfies the functional equation \cite{T}
\begin{equation}
\eta(s) = \lambda(s)\eta(1-s) \label{func1}
\end{equation}
where
\begin{equation}
\lambda(s)=\frac{1-2^{1-s}}{1-2^{s}}2^{s}\pi^{s-1}\sin\left(\frac{\pi s}{2} \right)\Gamma(1-s) \label{chi1}
\end{equation}
From this, one immediately obtains the means to extend the definition of the eta function to the whole complex plane. Indeed, Euler's acceleration of the conditionally convergent series in Eq. (\ref{eta1}) yields a double series that is absolutely and uniformly convergent on compact sets everywhere \cite{Sondow,Hasse,Blagouchine} 
\begin{equation}
\eta (s)=\sum_{k=0}^{\infty}\frac{1}{2^{k+1}}\sum_{m=0}^{k}{k \choose m}\frac{(-1)^{m}}{(m+1)^s} \label{Ser}
\end{equation}

The Dirichlet eta function is closely related to the Riemann zeta function by
\begin{equation}
\eta(s)=\left(1-2^{1-s}\right)\zeta(s) \label{etazeta}
\end{equation}
However, while the zeta function is meromorphic, with a pole at $s=1$, the eta function is an entire function. We note, from Eq. (\ref{eta1}), that at $s=1$, the eta function becomes the alternating harmonic series and, therefore,
\begin{equation}
\eta(1)=\sum_{m=1}^{\infty}\frac{(-1)^{m-1}}{m}=1-\frac{1}{2}+\frac{1}{3}-\ldots=\ln 2
\end{equation}

Let $s^*$ denote a zero of the eta function, $\eta(s^{*})=0$. There are two kinds of zeros: the trivial zeros for which, from the functional equation, we have $\lambda(s^{*})=0$; and the nontrivial zeros, for which $\eta(1-s^{*})= 0$. From Eq. (\ref{chi1}) the trivial zeros are the negative even integers and zeros of the form $\sigma^{*}=1+i\frac{2n\pi}{\ln 2}$ where $n$ is a nonzero integer (see \cite{Sondow2} for a derivation that does not make use of the functional relation).  

Since there are no zeros in the half-plane $\sigma >1$, the functional equation implies that nontrivial zeros of $\eta$ are to be found in the critical strip $0\le \sigma \le 1$. By the prime number theorem of Hadamard and de la Vall\'ee Poussin \cite{T} it is known that for $\sigma=1$ (and, therefore, $\sigma=0$) there are no nontrivial zeros of the Riemann zeta function and, therefore, from Eq. (\ref{etazeta}), there are nontrivial zeros of the Dirichlet eta function neither. Thus, the nontrivial zeros are found in the strip $0< \sigma < 1$. Furthermore, since, $\forall s \in \mathbb{C}$
\begin{equation}
\eta(\overline{s})=\overline{\eta(s)}
\end{equation}
where the line denotes complex conjugation, we have that $\overline{s}^{*}$ and $1-\overline{s}^{*}$ are also nontrivial zeros of $\eta$. In brief, nontrivial zeros come in quartets, $s^{*}$, $1-s^{*}$, $\overline{s^{*}}$ and $1-\overline{s^{*}}$ forming the vertices of a rectangle within the critical strip. The statement that, for the $\eta$ function, the nontrivial zeros have all real part $\sigma^{*}=1/2$, so that $s^{*}=1-\overline{s^{*}}$ and $\overline{s^{*}}=1-s^{*}$ (the rectangle degenerating in a line segment) is equivalent to the Riemann hypothesis for the Riemann zeta function \cite{Broughan}.

In this article we investigate the position of the nontrivial zeros of the eta function with help of nonlinear embeddings, a novel kind of mathematical structures introduced in our previous works \cite{JCOMPLEX,nembed}. We construct here a nonlinear embedding with the form of a double series that is absolutely and uniformly convergent on compact sets in the whole complex plane. We call this embedding a \emph{holomorphic nonlinear embedding}.  It depends on a scale parameter $\kappa \in (0, \infty)$ and a horizontal shift parameter $\nu \in (0,\infty)$, both in $\mathbb{R}$, and converges asymptotically to $\eta(s)$ everywhere as $\kappa$ tends to infinity. With help of M\"obius inversion, and taking advantage of absolute convergence of the series concerned, we then show the crucial property that  $\eta(s)$ itself can be expressed as a linear combination of horizontal shifts of the holomorphic nonlinear embedding $\eta_{\kappa \mu}(s)$ and we study the implications of this linear combination on the position of the nontrivial zeros of the eta function giving a proof of the Riemann hypothesis.

The outline of this article is as follows. In Section \ref{holocons} we construct   the holomorphic nonlinear embedding $\eta_{\kappa \nu}(s)$ for the Dirichlet eta function $\eta(s)$. We prove the global absolute and uniform convergence on compact sets of the series defining $\eta_{\kappa \nu}(s)$ and establish the asymptotic limits of the embedding. In Section \ref{crucial} we take advantage of these properties (specifically, we make heavy use of the absolute convergence of this series) to derive the crucial properties: 1) the holomorphic nonlinear embedding can be expressed as a linear combination of horizontally shifted eta functions; 2) the eta function itself can be expressed as a linear combination of horizontally shifted holomorphic nonlinear embeddings. These results are then exploited in Section \ref{RiemannH} to derive a functional relationship for the embedding and to prove the Riemann hypothesis, a result that emerges from the shift independence in the limit $\kappa \to \infty$ of the construction.

\section{Holomorphic nonlinear embedding for the Dirichlet eta function} \label{holocons}

We first introduce some notations and the basic functions on which our approach is based.

\begin{defi} Let $x\in \mathbb{R}$. We define the $\mathcal{B}_{\kappa}$-function \cite{JPHYSA} as
\begin{equation}
\mathcal{B}_{\kappa}(x):= \frac{1}{2}\left[\tanh\left(\frac{x+\frac{1}{2}}{\kappa}\right)-\tanh\left(\frac{x-\frac{1}{2}}{\kappa}\right)\right] \label{generbox}
\end{equation}
where $\kappa \in (0,\infty)$ is a real parameter.
\end{defi}

By noting that
\begin{eqnarray}
\mathcal{B}_{\kappa}(x)&=&\frac{e^{1/\kappa}-e^{-1/\kappa}}{e^{1/\kappa}+e^{2x/\kappa}+e^{-2x/\kappa}+e^{-1/\kappa}} \\
\frac{\mathcal{B}_{\kappa}\left(x\right)}{\mathcal{B}_{\kappa}\left(0\right)}&=&\frac{e^{1/\kappa}+2+e^{-1/\kappa}}{e^{1/\kappa}+e^{2x/\kappa}+e^{-2x/\kappa}+e^{-1/\kappa}} \label{easier}
\end{eqnarray}
the following properties are easily verified:
\begin{eqnarray}
0\le \mathcal{B}_{\kappa}\left(x \right) &\le & 1 \qquad \forall{\kappa \in(0,\infty)} \qquad \label{lim00} \\
\mathcal{B}_{\kappa}\left(-x \right) &= & \mathcal{B}_{\kappa}\left(x \right) \label{even} \\
\lim_{\kappa \to \infty}\mathcal{B}_{\kappa}\left(x \right)&=& 0 \label{lim0} \\
0\le \frac{\mathcal{B}_{\kappa}\left(x\right)}{\mathcal{B}_{\kappa}\left(0\right)} &\le & 1 \qquad \forall{\kappa \in(0,\infty)} \qquad \label{lim2} \\
\lim_{\kappa \to \infty}\frac{\mathcal{B}_{\kappa}\left(x\right)}{\mathcal{B}_{\kappa}\left(0\right)}&=&1 \label{lim1} 
\end{eqnarray}

\begin{defi} \emph{\textbf{(Holomorphic nonlinear embedding.)}} Let $\eta(s)$ be the Dirichlet eta function, given by Eq. (\ref{Ser}). Then, we define the holomorphic nonlinear embedding $\eta_{\kappa\nu}(s)$ of $\eta(s)$ as the series
\begin{equation}
\eta_{\kappa\nu}(s):=\sum_{k=0}^{\infty}\frac{1}{2^{k+1}}\sum_{m=0}^{k}{k \choose m}\frac{(-1)^{m}}{(m+1)^{s}}\frac{\mathcal{B}_{\kappa}\left(1/(m+1)^{\nu}\right)}{\mathcal{B}_{\kappa}\left(0\right)} \label{RK}
\end{equation}
with the real parameters $\kappa \in (0,\infty)$ and $\nu\in (0,\infty)$.
\end{defi}

\begin{theor} The double series in Eq. (\ref{RK}) converges absolutely and uniformly on compact sets to the entire function $\eta_{\kappa \nu}(s)$.
\end{theor}

\begin{proof}
We build on Sondow \cite{Sondow}, who proved that the double series defining the eta function in Eq. (\ref{Ser}) converge absolutely and uniformly on compact sets in the entire complex plane. In particular, if we define,
\begin{equation}
f_{k}(s):=\frac{1}{2^{k+1}}\sum_{m=0}^{k}{k \choose m}\frac{(-1)^{m}}{(m+1)^s}
\end{equation}
so that $\eta(s)=\sum_{k=0}^{\infty}f_{k}(s)$, Sondow proved that 
there is a sequence of positive real numbers $\{M_{k}\}$ satisfying 
\begin{equation}
\left|f_{k}(s)\right| \le \frac{1}{2^{k+1}}\sum_{m=0}^{k}\left|{k \choose m}\frac{(-1)^{m}}{(m+1)^{s}}\right| \le M_{k} \label{keyt2d1}
\end{equation}
and which 
\begin{equation}
\sum_{k=0}^{\infty} M_{k} <\infty
\end{equation}
so that Weierstrass M-test is satisfied. We now note that we can write $\eta_{\kappa\nu}(s)$ as
\begin{equation}
\eta_{\kappa\nu}(s)=\sum_{k=0}^{\infty}f_{k,\kappa\nu}(s)
\end{equation}
where
\begin{equation}
f_{k,\kappa\nu}(s):=\frac{1}{2^{k+1}}\sum_{m=0}^{k}{k \choose m}\frac{(-1)^{m}}{(m+1)^s} \frac{\mathcal{B}_{\kappa}\left(1/(m+1)^{\nu}\right)}{\mathcal{B}_{\kappa}\left(0\right)}
\end{equation}
Now, by the triangle inequality,
\begin{eqnarray}
|f_{k,\kappa\nu}(s)|&=&\left|\frac{1}{2^{k+1}}\sum_{m=0}^{k}{k \choose m}\frac{(-1)^{m}}{(m+1)^{s}}\frac{\mathcal{B}_{\kappa}\left(1/(m+1)^{\nu}\right)}{\mathcal{B}_{\kappa}\left(0\right)}\right| \nonumber \\
&&\le  \frac{1}{2^{k+1}}\sum_{m=0}^{k}\left|{k \choose m}\frac{(-1)^{m}}{(m+1)^{s}}\frac{\mathcal{B}_{\kappa}\left(1/(m+1)^{\nu}\right)}{\mathcal{B}_{\kappa}\left(0\right)}\right| \nonumber \\
&&=\frac{1}{2^{k+1}}\sum_{m=0}^{k}\left|{k \choose m}\frac{(-1)^{m}}{(m+1)^{s}}\right| \left|\frac{\mathcal{B}_{\kappa}\left(1/(m+1)^{\nu}\right)}{\mathcal{B}_{\kappa}\left(0\right)}\right| \nonumber \\
&&\le \frac{1}{2^{k+1}}\sum_{m=0}^{k}\left|{k \choose m}\frac{(-1)^{m}}{(m+1)^{s}}\right|
\end{eqnarray}
where Eq. (\ref{lim2}) has been used. Therefore, from this last expression and Eq. (\ref{keyt2d1})
\begin{equation}
\left|f_{k,\mu\nu}(s)\right| \le \frac{1}{2^{k+1}}\sum_{m=0}^{k}\left|{k \choose m}\frac{(-1)^{m}}{(m+1)^{s}}\right| \le M_{k} 
\end{equation}
and, thus, the sequence of positive real numbers $\{M_{k}\}$ found by Sondow majorizes the sequence $\{\left|f_{k,\mu\nu}(s)\right|\}$ as well, and the result follows.
\end{proof}

\begin{theor} We have, $\forall s\in \mathbb{C}$ 
\begin{eqnarray} 
\lim_{\kappa\to \infty}\eta_{\kappa \nu}(s)&=&\eta(s) \qquad \qquad \qquad \qquad \forall \nu \in (0,\infty) \label{t221} \\
\lim_{\kappa\to 0}\eta_{\kappa \nu}(s)&=&\eta(s)-1 \ \qquad \qquad \qquad \forall \nu \in (0,\infty) \label{t222} \\
\lim_{\nu\to \infty}\eta_{\kappa \nu}(s)&=&\eta(s)+\frac{\mathcal{B}_{\kappa}(1)}{\mathcal{B}_{\kappa}(0)}-1 \qquad \ \ \forall \kappa \in (0,\infty) \label{t223} \\
\lim_{\nu\to 0}\eta_{\kappa \nu}(s)&=&\frac{\mathcal{B}_{\kappa}(1)}{\mathcal{B}_{\kappa}(0)}\eta(s) \qquad \qquad \quad \ \ \forall \kappa \in (0,\infty) \label{t224} 
\end{eqnarray}
where $\mathcal{B}_{\kappa}(1)/\mathcal{B}_{\kappa}(0)=(\tanh \frac{3}{2\kappa}-\tanh \frac{1}{2\kappa})/(2\tanh \frac{1}{2\kappa})$, as given by Eq. (\ref{generbox}).
\end{theor}

\begin{proof} We first observe that, from Eq. (\ref{easier}),
\begin{eqnarray}
&&\lim_{\kappa \to \infty} \frac{\mathcal{B}_{\kappa}\left(1/(m+1)^{\nu}\right)}{\mathcal{B}_{\kappa}\left(0\right)}=1 \nonumber \\
&&\lim_{\kappa \to 0} \frac{\mathcal{B}_{\kappa}\left(1/(m+1)^{\nu}\right)}{\mathcal{B}_{\kappa}\left(0\right)}=
\left\{
\begin{array}{cc}
1  & \text{if } m\ge 1     \\
0  &  \text{if } m= 0       
\end{array}\right. \nonumber \\
&&\lim_{\nu \to \infty} \frac{\mathcal{B}_{\kappa}\left(1/(m+1)^{\nu}\right)}{\mathcal{B}_{\kappa}\left(0\right)}=
\left\{
\begin{array}{cc}
1  & \text{if } m\ge 1     \\
\frac{\mathcal{B}_{\kappa}\left(1\right)}{\mathcal{B}_{\kappa}\left(0\right)}  &  \text{if } m= 0       
\end{array}
\right. \nonumber  \\
&&\lim_{\nu \to 0} \frac{\mathcal{B}_{\kappa}\left(1/(m+1)^{\nu}\right)}{\mathcal{B}_{\kappa}\left(0\right)}=\frac{\mathcal{B}_{\kappa}\left(1\right)}{\mathcal{B}_{\kappa}\left(0\right)} \nonumber 
\end{eqnarray}

By using Eq. (\ref{RK}) and the above expressions, we get
\begin{eqnarray}  
\lim_{\kappa \to \infty}\eta_{\kappa\nu}(s)&=& \lim_{\kappa \to 0}\left[ \sum_{k=0}^{\infty}\frac{1}{2^{k+1}}\sum_{m=0}^{k}{k \choose m}\frac{(-1)^{m}}{(m+1)^{s}}\frac{\mathcal{B}_{\kappa}\left(1/(m+1)^{\nu}\right)}{\mathcal{B}_{\kappa}\left(0\right)}\right] \nonumber \\
&=&\sum_{k=0}^{\infty}\frac{1}{2^{k+1}}\sum_{m=0}^{k}{k \choose m}\frac{(-1)^{m}}{(m+1)^{s}}=\eta(s) \nonumber \\
\lim_{\kappa \to 0}\eta_{\kappa\nu}(s)&=& \sum_{k=0}^{\infty}\frac{1}{2^{k+1}}\sum_{m=1}^{k}{k \choose m}\frac{(-1)^{m}}{(m+1)^{s}}=\eta(s)-1 \nonumber \\ 
\lim_{\nu \to \infty}\eta_{\kappa\nu}(s)&=& \frac{\mathcal{B}_{\kappa}\left(1\right)}{\mathcal{B}_{\kappa}\left(0\right)}+\sum_{k=0}^{\infty}\frac{1}{2^{k+1}}\sum_{m=1}^{k}{k \choose m}\frac{(-1)^{m}}{(m+1)^{s}}=\eta(s)+\frac{\mathcal{B}_{\kappa}\left(1\right)}{\mathcal{B}_{\kappa}\left(0\right)}-1 \nonumber \\
\lim_{\nu \to 0}\eta_{\kappa\nu}(s)&=& \frac{\mathcal{B}_{\kappa}\left(1\right)}{\mathcal{B}_{\kappa}\left(0\right)}\sum_{k=0}^{\infty}\frac{1}{2^{k+1}}\sum_{m=0}^{k}{k \choose m}\frac{(-1)^{m}}{(m+1)^{s}}=\frac{\mathcal{B}_{\kappa}\left(1\right)}{\mathcal{B}_{\kappa}\left(0\right)}\eta(s).  \qedhere \nonumber
\end{eqnarray}
\end{proof}

\section{Functional expansions of $\eta_{\kappa \nu}(s)$ and $\eta(s)$} \label{crucial}

We now derive an equivalent expression for $\eta_{\kappa\nu}(s)$ valid for any $\kappa$ sufficiently large (specifically, $\forall \kappa > 3/\pi$) and prove that this expression can be inverted to express $\eta(s)$ as a function of horizontal shifts of $\eta_{\kappa\nu}(s)$.

\begin{theor} \label{1} If $\kappa>3/\pi$, $\forall \nu \in (0, \infty)$ the holomorphic embedding $\eta_{\kappa\nu}(s)$ has the absolutely convergent series expansion
\begin{eqnarray} 
\eta_{\kappa\nu}(s)&=&\eta(s)+\sum_{n=1}^{\infty}a_n(\kappa) \eta(s+2\nu n) 
 \label{asinto}
\end{eqnarray}
where, for $n$ a non-negative integer
\begin{eqnarray}
a_{n}(\kappa)&=&\frac{1}{\tanh \left(1/2\kappa \right)}\sum_{j=n+1}^{\infty}\frac{2^{2n}(2^{2j}-1)B_{2j}}{j(2j-2n-1)!(2n)!\kappa^{2j-1}}  \label{wjs}  
\end{eqnarray}
and $B_{2m}$ denote the even Bernoulli numbers: $B_{0}=1$, $B_{2}=\frac{1}{6}$, $B_{4}=-\frac{1}{30}$, etc.
\end{theor}

\begin{proof} For $\kappa > \frac{3}{\pi}$, the hyperbolic tangents in the definition of the $\mathcal{B}$-function, Eq. (\ref{generbox}) with $x=1/(m+1)^{\nu}$, can be expanded in their absolutely convergent MacLaurin series for all $\forall m \ge 0$ (note that  $\forall \nu > 0$, $1/(m+1)^{\nu} \le 1$)
\begin{eqnarray}
&&\mathcal{B}_{\kappa}\left(\frac{1}{(m+1)^{\nu}}\right)=\nonumber \\
&=&\frac{1}{2}\sum_{j=1}^{\infty}\frac{2^{2j}(2^{2j}-1)B_{2j}}{(2j)!\kappa^{2j-1}}\left[\left(\frac{1}{(m+1)^{\nu}}+\frac{1}{2} \right)^{2j-1}-\left(\frac{1}{(m+1)^{\nu}}-\frac{1}{2} \right)^{2j-1}\right] \nonumber \\
&=&  \frac{1}{2}\sum_{j=1}^{\infty}\frac{2^{2j}(2^{2j}-1)B_{2j}}{(2j)!\kappa^{2j-1}}
\sum_{h=0}^{2j-1}{2j-1 \choose h}\frac{1}{2^{h}(m+1)^{\nu(2j-1-h)}}(1-(-1)^{h}) \nonumber \\
&=&  \sum_{j=1}^{\infty}\frac{2(2^{2j}-1)B_{2j}}{(2j)! \kappa^{2j-1}}
\sum_{h=1}^{j}{2j-1 \choose 2h-1}\frac{2^{2(j-h)}}{(m+1)^{2(j-h)\nu}}   \nonumber \\
&=&  \sum_{j=1}^{\infty}\frac{2(2^{2j}-1)B_{2j}}{(2j)!\kappa^{2j-1}}
\sum_{n=0}^{j-1}{2j-1 \choose 2j-2n-1}\frac{2^{2n}}{(m+1)^{2n\nu}}  \nonumber \\
&=&  \sum_{n=0}^{\infty}\frac{1}{(m+1)^{2n\nu}} 
\sum_{j=n+1}^{\infty}{2j-1 \choose 2n} \frac{2^{2n+1}(2^{2j}-1)B_{2j}}{(2j)!\kappa^{2j-1}}
 \label{Bernoultheo1}
\end{eqnarray}
where we have used the absolute convergence of the series to change the order of the sums. 
We also have
\begin{eqnarray}
\mathcal{B}_{\kappa}\left(0\right)=\tanh \frac{1}{2\kappa}=\sum_{j=1}^{\infty}\frac{2(2^{2j}-1)B_{2j}}{(2j)! \kappa^{2j-1}} \label{Bernoultheo2} 
\end{eqnarray}
 
If we then replace these expansions in the definition of the embedding, Eq. (\ref{RK}), we find, by exploiting the absolute convergence of the series
\begin{eqnarray}
\eta_{\kappa\nu}(s)&=&\sum_{n=0}^{\infty}\sum_{k=0}^{\infty}\frac{1}{2^{k+1}}\sum_{m=0}^{k}{k \choose m}\frac{(-1)^{m}}{(m+1)^{s+2n\nu}} 
\sum_{j=n+1}^{\infty}{2j-1 \choose 2n} \frac{2^{2n+1}(2^{2j}-1)B_{2j}}{(2j)!\kappa^{2j-1}\mathcal{B}_{\kappa}\left(0\right)} \nonumber \\
&=&\sum_{n=0}^{\infty} \eta(s+2\nu n)
\sum_{j=n+1}^{\infty}{2j-1 \choose 2n} \frac{2^{2n+1}(2^{2j}-1)B_{2j}}{(2j)!\kappa^{2j-1}\mathcal{B}_{\kappa}\left(0\right)} \nonumber \\
&=&\sum_{n=0}^{\infty}a_n(\kappa) \eta(s+2\nu n) = \eta(s)+\sum_{n=1}^{\infty}a_n(\kappa) \eta(s+2\nu n) \label{RKp1}
\end{eqnarray}
where, for all non-negative integer $n$, we have defined
\begin{eqnarray}
a_n(\kappa) &:=& \sum_{j=n+1}^{\infty}\frac{2(2^{2j}-1)B_{2j}}{(2j)!\kappa^{2j-1}\mathcal{B}_{\kappa}\left(0\right)}
2^{2n}{2j-1 \choose 2n} \label{acoef1} \\
&=&\frac{1}{\tanh \left(1/2\kappa \right)}\sum_{j=n+1}^{\infty}\frac{2^{2n}(2^{2j}-1)B_{2j}}{j(2j-2n-1)!(2n)!\kappa^{2j-1}} \nonumber 
\end{eqnarray}
and we have also used that, from Eq. (\ref{Bernoultheo2})
\begin{eqnarray}
a_0&=&\sum_{j=1}^{\infty}\frac{2(2^{2j}-1)B_{2j}}{(2j)!\kappa^{2j-1}\mathcal{B}_{\kappa}\left(0\right)}
{2j-1 \choose 2j-1}=1.  \nonumber  \qedhere
\end{eqnarray}

\end{proof}

\begin{corol} Asymptotically, for $\kappa$ large, we have,
\begin{equation} 
\eta_{\kappa\nu}(s)=\eta(s)-\frac{\eta(s+2\nu)}{\kappa^2}+O\left(\frac{1}{\kappa^4} \right) 
 \label{asintostrong}
\end{equation}
\end{corol}

\begin{proof} From the theorem, we have $ \forall n\in \mathbb{Z}^{+}\cup \{0\}$ 
\begin{eqnarray}
a_{n}(\kappa)&=&\frac{1}{\tanh \left(1/2\kappa \right)}\sum_{j=n+1}^{\infty}\frac{2^{2n}(2^{2j}-1)B_{2j}}{j(2j-2n-1)!(2n)!\kappa^{2j-1}}  \nonumber \\
&=& \frac{1}{(\frac{1}{2\kappa}-O\left(\frac{1}{\kappa^3}\right) }\sum_{j=n+1}^{\infty}\frac{2^{2n}(2^{2j}-1)B_{2j}}{j(2j-2n-1)!(2n)!\kappa^{2j-1}}   \nonumber \\
&=& \left(1+O\left(\frac{1}{\kappa^2}\right)\right)\sum_{j=n+1}^{\infty}\frac{2^{2n+1}(2^{2j}-1)B_{2j}}{j(2j-2n-1)!(2n)!\kappa^{2j-2}}   \nonumber \\
&=&\frac{2^{2n+1}(2^{2n+2}-1)B_{2n+2}}{(n+1)(2n)!\kappa^{2n}}+ O\left(\frac{1}{\kappa^{2n+2}}\right)=O\left(\frac{1}{\kappa^{2n}}\right) \label{asina}
\end{eqnarray}
Therefore, 
\begin{equation}
\eta_{\kappa\nu}(s)=\eta(s)+30B_4\frac{\eta(s+2\nu)}{\kappa^2}+O\left(\frac{1}{\kappa^4} \right) 
\end{equation}
and the result follows from noting that $B_4=-1/30$.
\end{proof}

\begin{theor} \emph{\textbf{(M\"obius inversion formula.)}} \label{Moe}
For any $\kappa> 3/\pi$ and $\eta_{\kappa\nu}(s)$ given by Eq.(\ref{asinto}), 
\begin{equation}
\eta_{\kappa\nu}(s)=\eta(s)+\sum_{n=1}^{\infty}a_n(\kappa) \eta(s+2\nu n) \label{Gdir}
\end{equation}
we have,
\begin{equation}
\eta(s)=\eta_{\kappa\nu}(s)+\sum_{n=1}^{\infty}b_{n}(\kappa) \eta_{\kappa\nu}(s+2\nu n) \label{Ginv}
\end{equation}
where the coefficients $b_n$ are recursively obtained from
\begin{equation}
\sum_{n=0}^{k}b_{k}(\kappa)a_{n-k}(\kappa)=\delta_{k0} \label{bcoefs}
\end{equation}
with $\delta_{k0}$ being the Kronecker delta ($\delta_{k0}=1$ if $k=0$ and $\delta_{k0}=0$ otherwise).
\end{theor}

\begin{proof} We have
\begin{eqnarray}
\sum_{n=0}^{\infty}b_{n}\eta_{\kappa\nu}(s+2\nu n)
&=&\sum_{n=0}^{\infty}b_{n}\sum_{m=0}^{\infty}a_m(\kappa) \eta(s+2\nu n+2\nu m) \nonumber \\
&=&\sum_{k=0}^{\infty}\left[\sum_{n=0}^{k} b_{n}(\kappa)a_{k-n}(\kappa) \right] \eta(s+2\nu k) \nonumber \\
&=&\sum_{k=0}^{\infty}\delta_{k0} \eta(s+2\nu k) \nonumber\\
&=&\eta(s) \nonumber
\end{eqnarray}
The coefficients $b_n$ can be recursively obtained from the known $a_n$ by solving the equations
\begin{eqnarray}
a_0b_0&=&1 \\
a_0b_1+a_1b_0&=&0\\
a_0b_2+a_1b_1+a_2b_0&=&0\\
\ldots &\qquad & \nonumber
\end{eqnarray}
Where, since $a_0=1$, $b_0=1$, and, therefore, $b_1=-a_1$, $b_2=a_1^2-a_2$, etc.
\end{proof}

\begin{rem} Eq. (\ref{Ginv}), with $\eta_{\kappa \nu}$ given by Eq. (\ref{RK}) is the main result of this work, since it expresses the Dirichlet eta function in terms of a globally convergent series (absolutely and uniformly on compact sets) of horizontal shifts of $\eta_{\kappa \nu}$.  These shifts are weighted by powers of $1/\kappa$. 
\end{rem}

\begin{rem} Theorem \ref{Moe} is, indeed, a specific case of Theorem 3.3 on p. 82 in \cite{Nanxian} particularized to the functions $\eta$ and $\eta_{\kappa \nu}$ considered here.
\end{rem}

\begin{prop} For $\kappa >3/\pi$ we have
\begin{eqnarray}
\sum_{n=1}^{\infty}a_{n}(\kappa)&=&\frac{\tanh \frac{3}{2\kappa}-\tanh \frac{1}{2\kappa}}{2\tanh \frac{1}{2\kappa}}-1 \label{propo1} \\
\sum_{n=1}^{\infty}b_{n}(\kappa)&=&\frac{2\tanh \frac{1}{2\kappa}}{\tanh \frac{3}{2\kappa}-\tanh \frac{1}{2\kappa}}-1 \label{propo2}
\end{eqnarray}
\end{prop}

\begin{proof} From Eq. (\ref{t223})
\begin{eqnarray}
\lim_{\nu \to \infty}\eta_{\kappa\nu}(s)=\eta(s)+\frac{\mathcal{B}_{\kappa}\left(1\right)}{\mathcal{B}_{\kappa}\left(0\right)}-1 \label{coro21}
\end{eqnarray}
and, for $\kappa>3/\pi$, from Eqs. (\ref{Gdir}) and (\ref{Ginv})
\begin{eqnarray}
\lim_{\nu \to \infty}\eta_{\kappa\nu}(s)&=&\eta(s)+\sum_{n=1}^{\infty}a_n(\kappa) = \eta(s)-\frac{\mathcal{B}_{\kappa}\left(1\right)}{\mathcal{B}_{\kappa}\left(0\right)}\sum_{n=1}^{\infty}b_n(\kappa) \label{coro22}
\end{eqnarray}
because, from Eqs. (\ref{Ser}) and (\ref{RK}), $\lim_{\nu \to \infty} \eta(s+2\nu n)=1$ and $\lim_{\nu \to \infty} \eta_{\kappa \nu}(s+2\nu n)=\mathcal{B}_{\kappa}\left(1\right)/\mathcal{B}_{\kappa}\left(0\right)$ for every integer $n\ge 1$. By equating Eqs. (\ref{coro21}) and  (\ref{coro22}) and by using Eq. (\ref{generbox}), the result follows. 
\end{proof}

\section{On the nontrivial zeros of $\eta$} \label{RiemannH}

In this section we prove the Riemann hypothesis. We first introduce three lemmas, that establish a functional equation for the embedding and its asymptotic properties. 

\begin{lemma} \label{lem1}\emph{\textbf{(Functional relationship for the embedding.)}} We have $\forall \nu \in (0,\infty)$ and $\forall \kappa > \pi/3$ 
\begin{equation}
\eta_{\kappa\nu}(s)-\lambda(s)\eta_{\kappa\nu}(1-s)
=\sum_{n=1}^{\infty}b_{n}(\kappa) \left[\lambda(s)\eta_{\kappa\nu}(1-s+2\nu n)-\eta_{\kappa\nu}(s+2\nu n) \right] \label{funcembed}
\end{equation}
where 
\begin{equation}
\lambda(s)=\frac{1-2^{1-s}}{1-2^{s}}2^{s}\pi^{s-1}\sin\left(\frac{\pi s}{2} \right)\Gamma(1-s) \label{chi1b}
\end{equation}
\end{lemma}

\begin{proof} From Eq. (\ref{Ginv}) we have
\begin{eqnarray}
\eta(s)&=&\eta_{\kappa\nu}(s)+\sum_{n=1}^{\infty}b_{n}(\kappa) \eta_{\kappa\nu}(s+2\nu n) \nonumber \\
\lambda(s)\eta(1-s)&=&\lambda(s)\eta_{\kappa\nu}(1-s)+\lambda(s)\sum_{n=1}^{\infty}b_{n}(\kappa) \eta_{\kappa\nu}(1-s+2\nu n) \nonumber 
\end{eqnarray}
whence, by subtracting both equations and applying Eq. (\ref{func1}) the result follows.
\end{proof}

\begin{lemma} \label{lem2} If $s^{*}=\sigma^{*}+it^{*}$ is a non-trivial zero of $\eta(s)$ then $\eta_{\kappa \nu}(s^{*})\ne 0$ and $\eta_{\kappa \nu}(1-s^{*})\ne 0$ for finite asymptotically large $\kappa$ and $\nu > 1/2$. Furthermore, we have
\begin{equation}
 \lim_{\kappa \to \infty}\frac{\eta_{\kappa \nu}(s^{*})}{\eta_{\kappa \nu}(1-s^{*})}=\frac{\eta(s^{*}+2\nu)}{\eta(1-s^{*}+2\nu)} \label{limit1}
\end{equation}
\end{lemma}

\begin{proof} The real part $\sigma^{*}$ of the nontrivial zero $s^{*}$ satisfies $0< \sigma^{*} < 1$. Now, since $\eta(s^{*})=\eta(1-s^{*})=0$, we have, from Eq. (\ref{asintostrong})
\begin{eqnarray}
\eta_{\kappa\nu}(s^{*})&=&-\frac{\eta(s^{*}+2\nu)}{\kappa^2}+O\left(\frac{1}{\kappa^4} \right) \label{fund1} \\
\eta_{\kappa\nu}(1-s^{*})&=&-\frac{\eta(1-s^{*}+2\nu)}{\kappa^2}+O\left(\frac{1}{\kappa^4} \right) \label{fund2}
\end{eqnarray}
We have that $\forall \nu >1/2$, $\eta(s^{*}+2\nu)\ne 0$ and $\eta(1-s^{*}+2\nu)\ne 0$ because values $s^{*}+2\nu$ and $1-s^{*}+2\nu$ both lie in the half-plane $\sigma >1$ and $\eta(s)$ has no zeros there. Therefore $\eta_{\kappa\nu}(s^{*})\ne 0$ and $\eta_{\kappa\nu}(1-s^{*})\ne 0$ are both nonzero for sufficiently large $\kappa$ ($\kappa >3/\pi$ being a lower bound). Eq. (\ref{limit1}) follows then as a trivial consequence of Eqs. (\ref{fund1}) and (\ref{fund2}) and the absolute convergence of all the series involved. \end{proof}

\begin{lemma} Let $s_{\gamma}\in \mathbb{C}$, $s_{\gamma}\in \gamma$ be such that $\eta_{\kappa\nu}(1-s_{\gamma})\ne 0$ along a path $\gamma$ in the complex plane and let $s'$ be an endpoint of $\gamma$. Then,
\begin{equation}
\lim_{\kappa \to \infty}\lim_{s_{\gamma}\xrightarrow[\gamma]{} s'}\frac{\eta_{\kappa \nu}(s_{\gamma})}{\eta_{\kappa \nu}(1-s_{\gamma})}=\lim_{s_{\gamma}\xrightarrow[\gamma]{} s'}\lim_{\kappa \to \infty}\frac{\eta_{\kappa \nu}(s_{\gamma})}{\eta_{\kappa \nu}(1-s_{\gamma})}=\lambda(s') \label{commulim}
\end{equation}
Furthermore, if $s'=s^{*}$ is a nontrivial zero of the Dirichlet eta function, $\forall \nu > 1/2$
\begin{equation}
\frac{\eta(s^{*}+2\nu)}{\eta(1-s^{*}+2\nu)}=\lambda(s^{*}) \label{theotherhand}
\end{equation}  
\end{lemma}

\begin{proof} From Eq. (\ref{funcembed}) we have, by dividing by $\eta_{\kappa\nu}(1-s_{\gamma})$ at any point $s_{\gamma}$ of $\gamma$ (since $\eta_{\kappa\nu}(1-s_{\gamma})\ne 0$)
\begin{equation}
\frac{\eta_{\kappa\nu}(s_{\gamma})}{\eta_{\kappa\nu}(1-s_{\gamma})}=\lambda(s_{\gamma})
+\sum_{n=1}^{\infty}b_{n}(\kappa) \frac{\lambda(s_{\gamma})\eta_{\kappa\nu}(1-s_{\gamma}+2\nu n)-\eta_{\kappa\nu}(s_{\gamma}+2\nu n)}{\eta_{\kappa\nu}(1-s_{\gamma})} \label{funcembed}
\end{equation}
where $b_{1}(\kappa)=O(\kappa^{-2})$.  
We now have, on one hand,
\begin{eqnarray}
&&\lim_{\kappa \to \infty}\lim_{s_{\gamma}\xrightarrow[\gamma]{} s'}\frac{\eta_{\kappa \nu}(s_{\gamma})}{\eta_{\kappa \nu}(1-s_{\gamma})}=\lim_{\kappa \to \infty}\frac{\eta_{\kappa \nu}(s')}{\eta_{\kappa \nu}(1-s')} \label{side1} \\
&&=\lim_{\kappa \to \infty}\left[\lambda(s')+\sum_{n=1}^{\infty}b_{n}(\kappa) \frac{\lambda(s')\eta_{\kappa\nu}(1-s'+2\nu n)-\eta_{\kappa\nu}(s'+2\nu n)}{\eta_{\kappa\nu}(1-s')}\right] \nonumber \\
&&=\lim_{\kappa \to \infty}\left[\lambda(s')+\sum_{n=1}^{\infty}b_{n}(\kappa) \frac{\lambda(s')\eta(1-s'+2\nu n)-\eta(s'+2\nu n)}{\eta(1-s')}\right] \nonumber \\
&&=\lambda(s') \nonumber
\end{eqnarray}
and, on the other hand, 
\begin{eqnarray}
\lim_{s_{\gamma}\xrightarrow[\gamma]{} s'}\lim_{\kappa \to \infty}\frac{\eta_{\kappa \nu}(s_{\gamma})}{\eta_{\kappa \nu}(1-s_{\gamma})}&=&\lim_{s_{\gamma}\to s'}\frac{\eta(s_{\gamma})}{\eta(1-s_{\gamma})} \label{side2} \\
&=&\lim_{s_{\gamma}\to s'} \lambda(s_{\gamma})=\lambda(s') \nonumber
\end{eqnarray}
whence the result follows. 

Let us now assume that $s'=s^{*}$ is a nontrivial zero of the Dirichlet eta function. Then, we have that $\lim_{\kappa\to \infty}\eta_{\kappa \nu}(1-s^{*})=\eta(1-s^{*})=0=\eta(s^{*})=\lim_{\kappa\to \infty}\eta_{\kappa \nu}(s^{*})$ and the function in Eq. (\ref{side1})
\begin{equation}
\Phi(s'):=\lim_{\kappa \to \infty}\frac{\eta_{\kappa \nu}(s')}{\eta_{\kappa \nu}(1-s')}=\frac{\eta(s')}{\eta(1-s')}
\end{equation}
is undefined at $s'=s^{*}$. However, $s'=s^{*}$ is a removable singularity and we can take $\Phi(s^{*})=\lambda(s^{*})$. To see this, note that $\Phi(s)=\lambda(s)$ for all $s$ in the critical strip $0<\sigma<1$ save at the nontrivial zeros $s^{*}$. But the function $\lambda(s)$ is holomorphic for all $s$ in the critical strip including the nontrivial zeros $s^{*}$ of $\eta$. Therefore, by Riemann's theorem on extendable singularities, $\Phi(s)$ is holomorphically extendable over $s^*$ and we can have, in consistency with Eq. (\ref{side2})
\begin{equation}
\Phi(s^{*})=\lambda(s^{*}) \qquad \left(=\lim_{s\to s^{*}}\frac{\eta(s)}{\eta(1-s)} \right) 
\end{equation}
This proves Eq. (\ref{commulim}). We then note that, on one hand
\begin{equation}
\Phi(s^{*})=\lim_{\kappa \to \infty}\frac{\eta_{\kappa \nu}(s^{*})}{\eta_{\kappa \nu}(1-s^{*})}=\lambda(s^{*}) \label{the1}
\end{equation}
and, on the other, from Eq. (\ref{limit1}) 
\begin{equation}
\Phi(s^{*})= \lim_{\kappa \to \infty}\frac{\eta_{\kappa \nu}(s^{*})}{\eta_{\kappa \nu}(1-s^{*})}=\frac{\eta(s^{*}+2\nu)}{\eta(1-s^{*}+2\nu)} \label{the2}
\end{equation}
Both expressions must be equal at $s^{*}$ because: 1) Eq. (\ref{the1}) is a consequence of  $\Phi(s)$ being equal to the holomorphic $\lambda(s)$ in the punctured critical strip (save, exactly at the zeros $s^{*}$) and, therefore, holomorphically extendable over $s^{*}$ and 2) Eq. (\ref{the2}) is a consequence of the asymptotic behavior of the embedding close to a nontrivial zero of the Dirichlet eta function. Thus, Eq. (\ref{theotherhand}) follows.
\end{proof}

\noindent \emph{Alternative proof of Eq. (\ref{theotherhand})}. An equivalent way of obtaining Eq. (\ref{theotherhand}) is, directly, from Eq. (\ref{side1}), applying it to $s'=s^{*}$. We have,
\begin{eqnarray}
&&\lim_{\kappa \to \infty}\lim_{s_{\gamma}\xrightarrow[\gamma]{} s^{*}}\frac{\eta_{\kappa \nu}(s_{\gamma})}{\eta_{\kappa \nu}(1-s_{\gamma})}=\lim_{\kappa \to \infty}\frac{\eta_{\kappa \nu}(s^{*})}{\eta_{\kappa \nu}(1-s^{*})}  \\
&&=\lim_{\kappa \to \infty}\left[\lambda(s^{*})+\sum_{n=1}^{\infty}b_{n}(\kappa) \frac{\lambda(s^{*})\eta_{\kappa\nu}(1-s^{*}+2\nu n)-\eta_{\kappa\nu}(s^{*}+2\nu n)}{\eta_{\kappa\nu}(1-s^{*})}\right] \nonumber \\
&&=\lambda(s^{*})+\lim_{\kappa\to \infty}\frac{\lambda(s^{*})\eta_{\kappa\nu}(1-s^{*}+2\nu)-\eta_{\kappa \nu}(s^{*}+2\nu)}{\kappa^{2}\eta_{\kappa\nu}(1-s^{*})} \nonumber \\
&&=\lambda(s^{*})-\frac{\lambda(s^{*})\eta(1-s^{*}+2\nu)-\eta(s^{*}+2\nu)}{\eta(1-s^{*}+2\nu)} \label{onealt}
\end{eqnarray}
and since $\eta(1-s^{*}+2\nu) \ne 0$ $\forall \nu \in (1/2,\infty)$ and we have for any $\varepsilon \in \mathbb{C}$ in a disk of sufficiently small radius 
\begin{equation}
\Phi(s^{*}+\varepsilon)=\lim_{\kappa \to \infty}\frac{\eta_{\kappa \nu}(s^{*}+\varepsilon)}{\eta_{\kappa \nu}(1-s^{*}+\varepsilon)}=\lambda(s^{*}+\varepsilon) \label{twoalt}
\end{equation}
by taking the limit $\varepsilon \to 0$ and observing that $\Phi(s)$ is holomorphically extendable to $s^{*}$ where it has then the value $\lambda(s^{*})$ we obtain, from Eq. (\ref{onealt}) 
\begin{equation}
\lambda(s^{*})\eta(1-s^{*}+2\nu)-\eta(s^{*}+2\nu)=0
\end{equation}
which is Eq. (\ref{theotherhand}). $\square$

\begin{theor}  \emph{\textbf{(Riemann hypothesis.)}} \label{RHtheor} All nontrivial zeros $s^{*}=\sigma^{*}+it^{*}$ of the Dirichlet eta function $\eta(s)$  have real part $\sigma^{*}=1/2$.
\end{theor}

\begin{proof}  Let $\sigma^{*}=\frac{1}{2}+\epsilon$ be the real part of a nontrivial zero $s^{*}=\sigma^{*}+it^{*}$ of $\eta$ in the critical strip. From Eq. (\ref{theotherhand}), 
\begin{equation}
\left|\frac{\eta\left(\frac{1}{2}+2\nu+it^{*}+\epsilon\right)}{\eta\left(\frac{1}{2}+2\nu-it^{*}-\epsilon\right)}\right|=\left|\lambda\left(\frac{1}{2}+it^{*}+\epsilon\right)\right|, \label{evenga}
\end{equation} 
and we note that the right hand side of this expression is shift-invariant (it does not depend on $\nu$) but the left hand side is not: the horizontal shift parameter $\nu$ can be arbitrarily varied in the interval $(0,\infty)$ and, in particular, it can be selected so that the point $s^{*}+2\nu$ lies anywhere on the half-plane $\sigma\ge 1$ at height $t^{*}$. The modulus of $\eta$ varies on horizontal lines \cite{Matiyasevich0}. \emph{The only possibility for equation Eq. (\ref{evenga}) to have solution for a nontrivial zero $s^{*}$ forces $\epsilon=0$}. To see this, put $x=\frac{1}{2}+2\nu-\epsilon >>1$ in Eq. (\ref{evenga}). Now, since $\left|\eta\left(x-it^{*}\right)\right|=\left|\overline{\eta\left(x+it^{*}\right)}\right|=\left|\eta\left(x+it^{*}\right)\right|$, we have 
\begin{eqnarray}
\left|\frac{\eta\left(\frac{1}{2}+2\nu+it^{*}+\epsilon\right)}{\eta\left(\frac{1}{2}+2\nu-it^{*}-\epsilon\right)}\right|&=&\left|\frac{\eta\left(x+it^{*}+2\epsilon\right)}{\eta\left(x+it\right)}\right|.
\end{eqnarray}
Since $x$ can be increased arbitrarily by increasing $\nu$, we can take $x$ so large that, asymptotically
\begin{equation}
\eta\left(x+it^{*}+2\epsilon\right)\sim \eta\left(x+it^{*}\right)+2\epsilon \left.\frac{\partial \eta}{\partial x}\right|_{x+it^{*}}.
\end{equation}
In this asymptotic regime, we can truncate Eq. (\ref{eta1}) to the first two terms and its derivative becomes
\begin{equation}
\left.\frac{\partial \eta}{\partial x}\right|_{x+it^{*}} \sim \frac{\ln 2}{2^{x+it^{*}}}.
\end{equation}
Furthermore, $\eta\left(x+it^{*}\right)\sim 1-2^{-x-it^{*}}$ and thus
\begin{equation}
\left|\frac{\eta\left(x+it^{*}+2\epsilon\right)}{\eta\left(x+it\right)}\right| \sim \left | 1+2\epsilon \frac{\ln 2}{2^{x+it^{*}}-1}\right|.
\end{equation}
Therefore, for $\nu$ large,
\begin{equation}
\left|\frac{\eta\left(\frac{1}{2}+2\nu+it^{*}+\epsilon\right)}{\eta\left(\frac{1}{2}+2\nu-it^{*}-\epsilon\right)}\right| \sim  \left | 1+2\epsilon \frac{\ln 2}{2^{\frac{1}{2}+2\nu-\epsilon+it^{*}}-1}\right|,
\end{equation}
and the l.h.s. of Eq. (\ref{evenga}) depends explicitly on the free parameter $\nu$. Thus, the r.h.s. of Eq. (\ref{evenga}) can have an infinite number of different values for its modulus, which is absurd. The only possibility of cancelling the $\nu$ dependence, forced by the consistency of the equation, is to have $\epsilon=0$. In this way, we obtain,
\begin{equation}
\left|\frac{\eta\left(\frac{1}{2}+2\nu+it^{*}\right)}{\eta\left(\frac{1}{2}+2\nu-it^{*}\right)}\right|=1=\left|\lambda\left(\frac{1}{2}+it^{*}\right)\right|,
\end{equation}
an equation that is known to have infinitely many solutions for $t^{*}$.
Therefore, $\epsilon=0$, $s^{*}=\frac{1}{2}+it^{*}$ and the result follows. 
\end{proof}

\section{Conclusions}

In this article a complex entire function called a holomorphic nonlinear embedding $\eta_{\kappa \nu}$ has been constructed and some of its properties (mainly asymptotic ones) have been investigated. It has been shown that the function $\eta_{\kappa \nu}$ can be expressed as a series expansion in terms of horizontal shifts of the Dirichlet eta function. The holomorphic character of $\eta_{\kappa \nu}$ has been established by proving the global absolute and uniform convergence on compact sets of its defining double series. The coefficients of the expansion and their sum have been explicitly calculated. It has also been shown that this expansion can be inverted to yield the eta function as a linear expansion of horizontal shifts of $\eta_{\kappa \nu}$. This is the central result of this article, since it allows the Dirichlet eta function to be understood as a linear superposition of different layers governed by shifts of $\eta_{\kappa \nu}$ and weighted by powers of $1/\kappa$. This result shows that, although one can envisage, in principle, infinitely many ways of smoothly embedding the Dirichlet eta function in a more general structure, the one presented here ($\eta_{\kappa \nu}$) is not gratuitous because it is itself embedded within the structure of the eta function, revealing some of its secrets thanks to its scale and horizontal shift parameters $\kappa$ and $\nu$. In particular, the truth of the Riemann hypothesis emerges naturally as a consequence of the functional relationship of the Dirichlet eta function and its uniform attainment everywhere by a hierarchy of functional equations of the holomorphic nonlinear embedding in the limit $\kappa \to \infty$.

Operators yielding vertical shifts of the Riemann zeta function $\zeta(s)$ arise naturally in approaches to the Riemann zeros using ideas from supersymmetry \cite{DasKalauni}. These vertical shifts can, indeed, be expressed more compactly in terms of the Dirichlet eta function $\eta$ (see e.g. Eq.(18) in \cite{DasKalauni}). Whether there exists any relationship of these vertical shifts induced by lowering and raising operators in \cite{DasKalauni} with the shifts obtained here as a result of an explicit series construction is an interesting open question. We point out that $\nu$ in this article can be made a complex number with positive real part and arbitrary imaginary part and all main results of this article apply without any modification: Eqs. (\ref{Gdir}) and (\ref{Ginv}) are indeed valid, for $\nu=\nu_{r}+i\nu_{i}$, with $\nu_{r}>0$ and any $\nu_{i} \in (-\infty, \infty)$ and this does not affect the proof of the Riemann hypothesis here presented in any way. The condition $\nu_{r}>0$ is, however, necessary, for Eqs. (\ref{Gdir}) and (\ref{Ginv}) to be valid. 

The methods used here may be adapted to other Dirichlet series for which the Riemann hypothesis is conjectured to hold \cite{Karatsuba}. All steps may be retraced for these latter series if: 1) they can be closely related to entire functions with series expansions that are absolutely convergent in the whole complex plane; 2) the latter series depend on the complex variable $s$ only through summands of the form $1/m^{s}$; 3) there exists a functional relation like Eq. (\ref{func1}).  

We believe that the methods described in this paper might be useful to get insight in recent intriguing experimental phenomena connecting the coefficients of truncated Dirichlet series to the Erathostenes sieve \cite{Matiyasevich}.

\end{document}